\documentclass[11pt]{amsart}
\usepackage{amsmath, amsthm, amsfonts, amssymb}
\usepackage{geometry}
\usepackage{enumerate}

\theoremstyle{plain}
\newtheorem{thm}{Theorem}[section]

\newtheorem*{thma}{Theorem A}
\newtheorem*{thmb}{Theorem B}
\newtheorem*{thmc}{Theorem C}
\newtheorem{lem}[thm]{Lemma}
\newtheorem{prop}[thm]{Proposition}
\newtheorem{cor}[thm]{Corollary}

\theoremstyle{definition}
\newtheorem{defn}[thm]{Definition}

\newtheorem{exmp}[thm]{Example}

\newtheorem{rem}[thm]{Remark}
\newtheorem{rems}[thm]{Remarks}

\newcommand{\N}{\mathbb{N}}
\newcommand{\R}{\mathbb{R}}
\newcommand{\C}{\mathbb{C}}

\newcommand{\Stab}{\operatorname{Stab}}
\newcommand{\Norm}{\operatorname{Norm}}
\newcommand{\spa}{\operatorname{span}}

\newcommand{\ootimes}{\, \overline{\otimes} \,}
\newcommand{\Aut}{\operatorname{Aut}}
\newcommand{\ev}{\operatorname{ev}}
\newcommand{\odd}{\operatorname{odd}}
\newcommand{\Ind}{\operatorname{Ind}}

\begin{document}

\title[Solid ergodicity for Gaussian actions]{On solid ergodicity for Gaussian actions}

\begin{abstract}
We investigate Gaussian actions through the study of their crossed-product von Neumann algebra. The motivational result is Chifan and Ioana's ergodic decomposition theorem for Bernoulli actions (\cite{CI}) that we generalize to Gaussian actions (Theorem A). We also give general structural results (Theorems \ref{adapted IPV} and \ref{ultraproduct}) that allow us to get a more accurate result at the level of von Neumann algebras. More precisely, for a large class of Gaussian actions $\Gamma \curvearrowright X$, we show that any subfactor $N$ of $L^\infty(X) \rtimes \Gamma$ containing $L^\infty(X)$ is either hyperfinite or is non-Gamma and prime. At the end of the article, we show a similar result for Bogoliubov actions.
\end{abstract}

\author{Remi Boutonnet}

\address{ENS Lyon \\
UMPA UMR 5669 \\
69364 Lyon cedex 7 \\
France}

\email{remi.boutonnet@ens-lyon.fr}

\maketitle

\section{Introduction}

During the past few years, the strategy of using von Neumann algebras to study probability measure preserving (\emph{p.m.p.}) actions (or more generally p.m.p. equivalence relations) has led to several breakthroughs. This fact is mainly due to the deformation/rigidity technology developed by Popa (\cite{popasup,popamal1,popamal2,popamsri}) in order to study finite von Neumann algebras.\\
Crossed-product von Neumann algebras fit well in the context of deformation/rigidity, especially when the action involved is a Bernoulli type action. Indeed, Bernoulli actions admit nice deformation properties, being s-malleable in the sense of Popa (\cite{popasup}), but also a very strong mixing property. Thus these actions have been intensively studied and many deep results have been discovered (\emph{see} \cite{CI,ioana5,ipv} for example).\\
Another class of s-malleable actions is Gaussian actions. Recall that if $\Gamma$ is a countable group and $\pi : \Gamma \rightarrow \mathcal{U}(H)$ is a unitary representation of $\Gamma$, there exist (see \cite{ps} for instance) a standard probability space $(X,\mu)$ and a pmp action of $\Gamma$ on $X$, such that $H \subset L^2(X)$, as representations of $\Gamma$. This action is called the {\it Gaussian action} induced by the representation $\pi$.

Although Gaussian actions are not as mixing as Bernoulli actions, we show that some results about Bernoulli actions can be generalized. This will be the case of the following theorem.

\begin{thm}[Chifan-Ioana, \cite{CI}]
\label{CI}
Let $\Gamma \curvearrowright I$ be an action of a discrete countable group on a countable set $I$, with amenable stabilizers (\emph{i.e.} $Stab(i)= \lbrace g \in \Gamma , \, g \cdot i = i \rbrace$ is amenable for all $i \in I$).
Consider the generalized Bernoulli action $\Gamma \curvearrowright ([0,1],\text{Leb})^I$ given by $g \cdot (x_i)_i = (x_{g^{-1} \cdot i})_i$ and $\mathcal{R}_\Gamma^I = \mathcal{R}(\Gamma \curvearrowright ([0,1],\text{Leb})^I)$ the associated equivalence relation.\\
Then $\mathcal R_\Gamma^I$ is \emph{solidly ergodic}\footnote{Terminology introduced by Gaboriau in \cite[Section 5]{gaboriauICM}.}, that is,  $\mathcal R_\Gamma^I$ has the following property : \\
``For any sub-equivalence relation $\mathcal{R} \subset \mathcal{R}_\Gamma^I$, there exists a countable partition $X = \bigsqcup_{n \in \mathbb{N}} X_n$ of $X$ into measurable $\mathcal{R}$-invariant subsets with :
\begin{itemize}
\item $\mathcal{R}_{|X_0}$ hyperfinite ;
\item $\mathcal{R}_{|X_n}$ is strongly ergodic for all $n \geq 1$."
\end{itemize}
Moreover, a similar decomposition applies for any quotient relation of $\mathcal{R}$.
\end{thm}

Recall that a pmp equivalence relation on $(X,\mu)$ is said to be strongly ergodic if for any asymptotically invariant sequence $(A_n)$ of measurable subsets of $X$, one has $\lim_n \mu(A_n)(1-\mu(A_n)) = 0$. Also a pmp equivalence relation $\mathcal S$ on a space $X'$ is a quotient of a pmp relation $\mathcal R$ on $X$ if there exists an onto pmp Borel map $p : X \rightarrow X'$ such that $\mathcal{S} = p^{(2)}(\mathcal{R})$, where $p^{(2)}(x, y) = (p(x), p(y))$.\\

As Chifan and Ioana explained in their paper, Theorem \ref{CI} is related to Gaboriau and Lyons's theorem on von Neumann's problem about non-amenable groups\footnote{Von Neumann's problem asks whether every non-amemable group contains a copy of a free group or not.} : In \cite{GL}, Gaboriau and Lyons gave a positive answer to von Neumann's problem in the measurable setting. It turns out that one of the main steps of their proof can be deduced from Theorem \ref{CI}. For a survey on that topic, see \cite{houdayerbourbaki}.\\

To prove Theorem \ref{CI}, Chifan and Ioana showed \cite[Proposition 6]{CI} that a measure-preserving equivalence relation on a probability space $(X,\mu)$ is solidly ergodic if and only if $Q' \cap L\mathcal{R}$ is amenable for any diffuse subalgebra $Q \subset L^\infty(X,\mu)$.
Here $L\mathcal R$ denotes the von Neumann algebra associated to $\mathcal R$ (\cite{FMII}).

With the same strategy, we will prove the analogous result for Gaussian actions, with reasonable restrictions on the representation we start with.

\begin{defn}[Vaes, \cite{vaesonecoho}]
A representation $\pi : \Gamma \curvearrowright \mathcal{O}(H)$ of a discrete countable group $\Gamma$ is said to be {\it mixing relative} to a family $\mathcal S$ of subgroups of $\Gamma$ if for all $\xi,\eta \in H$ and $\varepsilon > 0$, there exist $g_1,\cdots,g_n,h_1,\cdots,h_n \in \Gamma$ and $\Sigma_1, \cdots, \Sigma_n \in \mathcal{S}$ such that $\vert \langle \pi(g)\xi,\eta \rangle \vert < \varepsilon$, for all $g \in \Gamma \setminus \cup_{i=1}^n g_i\Sigma_ih_i$.
\end{defn}

\begin{thma}
Let $\pi : \Gamma \rightarrow \mathcal{O}(H_{\mathbb{R}})$ be an orthogonal representation of a countable discrete group  $\Gamma$. Denote by $\Gamma \curvearrowright (X,\mu)$ the Gaussian action associated to $\pi$, and by $\mathcal{R}_{\pi}$ the corresponding equivalence relation on $X$.
Assume that the following two conditions hold :
\begin{enumerate}
\item Some tensor power of $\pi$ is tempered (\emph{meaning} weakly contained in the regular representation) ;
\item $\pi$ is mixing relative to a family $\mathcal{S}$ of amenable subgroups of $\Gamma$.
\end{enumerate}
Then $\mathcal{R}_{\pi}$ is solidly ergodic.
\end{thma}

Under an extra mixing condition on $\pi$, we get more accurate result on the the strongly ergodic pieces that appear in solid ergodicity. Moreover, we prove that a similar decomposition applies to more general algebras than algebras coming from subequivalence relations. First, we define a weak version of malnormality.

\begin{defn}
\label{malnormal}
A subgroup $\Sigma$ of a group $\Lambda$ is said to be $n$-almost malnormal ($n \geq 1$), if for any $g_1,\cdots,g_n \in \Lambda$ such that $g_i^{-1}g_j \notin \Sigma$ for all $i \neq j$, the subgroup $\cap_{i=1}^n g_i\Sigma g_i^{-1}$ is finite. It is said to be almost-malnormal if it is $n$-almost malnormal for some $n \geq 1$.
\end{defn}

\begin{thmb}
Assume that condition (1) of Theorem A holds and that $\pi$ is mixing relative to a finite family $\mathcal S$ of amenable, almost-malnormal subgroups of $\Gamma$. Denote by $M = L^\infty(X) \rtimes \Gamma$ the crossed-product von Neumann algebra of the Gaussian action $\Gamma \curvearrowright (X,\mu)$  associated to $\pi$.\\
Let $Q \subset M$ be a subalgebra such that $Q \nprec_M L\Gamma$. Then there exists a sequence $(p_n)_{n \geq 0}$ of projections in $\mathcal Z(Q)$ with $\sum_n p_n = 1$ such that :
\begin{itemize}
\item $p_0Q$ is hyperfinite ;
\item $p_nQ$ is a prime factor and does not have property Gamma.
\end{itemize}
\end{thmb}

The following classes of representations satisfy the conditions of the two theorems above :
\begin{itemize}
\item Quasi-regular representations $\Gamma \curvearrowright \ell^2(\Gamma / \Sigma)$ with $\Sigma < \Gamma$ amenable and almost malnormal. Indeed, if $\Sigma$ is amenable one checks that the associated quasi-regular representation is tempered.
As explained in Example \ref{bernoulli}, in this case the associated Gaussian action is the generalized Bernoulli shift. Hence, Theorem A is indeed a generalization of Theorem \ref{CI}.
\item Strongly $\ell^p$ representations\footnote{A representation $\pi$ on $H$ is said to be strongly $\ell^p$ if for all $\varepsilon > 0$, there exists a dense subspace $H_0\subset H$ such that for all $\xi,\eta \in H_0$, $(\langle \pi(g)\xi,\eta \rangle) \in \ell^{p+\varepsilon}(\Gamma)$, \cite{shalomrigidity}} with $p \geq 2$. Sinclair pointed out in \cite{sinclair1} (using \cite{{CHH},{shalomrigidity}}) that these representations admit a tensor power which is tempered, and they are clearly mixing.
\end{itemize}
As we will see in section 2.3, if the representation we start with is strongly $\ell^p$ for $p > 2$, but not tempered, then the associated Gaussian action is \emph{not} a Bernoulli action.\\

At the end of the article, we prove the following adaptation of Theorem B in the context of Bogoliubov actions on the hyperfinite ${\rm II_1}$ factor (see section 5 for details).

\begin{thmc}
Assume that the representation $\pi$ is mixing relative to a finite family $\mathcal{S}$ of almost-malnormal amenable subgroups of $\Gamma$ and has a tensor power which is tempered.
Consider the Bogoliubov action $\Gamma \curvearrowright R$ on the hyperfinite II$_1$ factor associated to $\pi$, and put $M = R \rtimes \Gamma$.\\
Let $Q \subset M$ be a subalgebra such that $Q \nprec_M L\Gamma$. Then there exists a sequence  $(p_n)_{n \geq 0}$ of projections in $\mathcal Z(Q)$ with $\sum_n p_n = 1$ such that :
\begin{itemize}
\item $p_0Q$ is hyperfinite ;
\item $p_nQ$ is a prime factor and does not have property Gamma.
\end{itemize}
\end{thmc}

\subsection*{About the proofs of the main theorems} The proofs of Theorems A and B (and C as well) rely on a localization theorem (Theorem \ref{adapted IPV}) for subalgebras in the crossed-product $M = L^\infty(X) \rtimes \Gamma$, in the spirit of Theorem 5.2 of \cite{popasup}. In fact this is a generalization of Theorem 4.2 in \cite{ipv}, and the proof follows the same lines. Theorem A will be an immediate consequence of that result (modulo a spectral gap argument), whereas Theorem B will require more work on the ultraproduct von Neumann algebra of $M$ (Theorem \ref{ultraproduct}).

\subsection*{Structure of the article} Aside from the introduction this article contains 4 other sections. Section 2 is devoted to preliminaries about Gaussian actions and intertwining techniques. In Section 3, we use deformation/rigidity techniques to locate rigid subalgebras in the crossed-product or in its ultraproduct (Theorems \ref{adapted IPV} and \ref{ultraproduct}). In section 4, we prove Theorems A and B. The proof of Theorem C is presented in an extra-section, devoted to Bogoliubov actions.

\subsection*{Acknowledgement}
We are very grateful to Cyril Houdayer for suggesting this problem, and for all the great discussions and advice that he shared with us. We also thank Bachir Bekka for explaining Proposition \ref{nonbernoulli2} to us.

\section{Preliminaries}

\subsection{Terminology, notations and conventions}

In this article, all finite von Neumann algebras are equipped with a distinguished faithful normal trace $\tau$.\\ 
Every action of a discrete countable group $\Gamma$ on $M$ is assumed to preserve the trace, and $M \rtimes \Gamma$ denotes the associated crossed-product von Neumann algebra, equipped with the trace defined by $\tau(xu_g) = \tau_M(x)\delta_{g,e}$, for all $x \in M$, $g \in \Gamma$.\\
If $M$ is a finite von Neumann algebra, denote by $L^2(M)$ the GNS construction of $M$ for its distinguished trace. For a subspace $H \subset L^2(M)$, put $H^* = JH$, where $J : L^2(M) \rightarrow L^2(M)$ is the anti-linear involution defined by $x \mapsto x^*$, for $x \in M$.\\
If $Q \subset M$ are finite von Neumann algebras, the distinguished trace on $Q$ is obviously the restriction of the distinguished trace on $M$, and we write $E_Q: M \rightarrow Q$ for the unique trace-preserving conditional expectation onto $Q$ and $e_Q : L^2(M) \rightarrow L^2(Q)$ for the corresponding projection. Also, $\mathcal U(M)$ refers to the group of unitary elements in $M$, and $\mathcal{N}_M(Q) = \lbrace u \in \mathcal{U}(M) \, \vert \, uQu^* = Q \rbrace$ denotes the normalizer of $Q$ in $M$.\\
If $P,Q \subset M$ are von Neumann algebras, an element $x \in M$ is said to be $P$-$Q$-finite if there exist $x_1,\cdots,x_n,y_1,\cdots,y_m \in M$ such that 
$$xQ \subset \sum_{i=1}^n Px_i, \; \text{and} \; Px \subset \sum_{j=1}^m y_jQ.$$
The {\it quasi-normalizer} of $Q \subset M$ is the set of $Q$-$Q$-finite elements in $M$, and is denoted $\mathcal{QN}_M(Q)$.\\

Finally, given a von Neumann algebra $M$ and two $M$-$M$ bimodules $H$ and $K$, we write $H \subset_w K$ to denote that $H$ is weakly contained in $K$.
Moreover if $H$ and $K$ are two $M$-$M$ bimodules, we denote by $H \otimes_M K$ the Connes fusion tensor product of $H$ and $K$ (\cite{popacor}). If $\xi \in H$ is a right bounded vector, and $\eta \in K$, the element of $H \otimes_M K$ corresponding to $\xi \otimes \eta$ is denoted $\xi \otimes_M \eta$.

\subsection{Popa's intertwining technique}

We recall in this section one of the main ingredients of Popa's deformation/rigidity strategy : intertwining by bimodule.

\begin{thm}[Popa, \cite{{popamal1},{popamsri}}]
\label{intertwining}
Let $P,Q \subset M$ be finite von Neumann algebras and assume that $Q \subset M$ is a unital inclusion. Then the following are equivalent.
\begin{itemize}
\item There exist projections $p \in P$, $q \in Q$, a normal $*$-homomorphism $\psi: pPp \rightarrow qQq$, and a non-zero partial isometry $v \in pMq$ such that $xv = v\psi(x)$, for all $x \in pPp$ ;
\item There exists a $P$-$Q$ subbimodule $H$ of $L^2(1_PM)$ which has finite index when regarded as a right $Q$-module ;
\item There is no sequence of unitaries $(u_n) \in \mathcal U(P)$ such that $\Vert E_Q(x^*u_ny) \Vert_2 \rightarrow 0$, for all $x,y \in M$.
\end{itemize}
\end{thm}

Following \cite{popamal1}, if $P,Q \subset M$ satisfy these conditions, we say that {\it a corner of $P$ embeds into $Q$ inside $M$}, and we write $P \prec_M Q$.\\
Note that there also exists a ``diagonal version" of this theorem (see Remark 3.3 in \cite{vaesbimodule} for instance) : If $(Q_k)$ is a sequence of subalgebras of $M$ such that $P \nprec_M Q_k$ for all $k$, then one can find a sequence of unitaries $u_n \in \mathcal{U}(P)$ such that $\lim_n \Vert E_{Q_k}(xu_ny)\Vert_2 = 0$, for all $k \in \N$.

We also mention a relative version\footnote{meaning {\it relative to a subspace of $L^2(M)$}} of Theorem \ref{intertwining}.

\begin{lem}[Vaes, \cite{vaesbimodule}]
\label{relativeintertwining}
Let $B \subset M$ be finite von Neumann algebras, and $H \subset L^2(M)$ a $B$-$B$ sub-bimodule. Assume that there exists a sequence of unitaries $u_n \in \mathcal U(B)$ such that $$\lim_n \Vert e_B(xu_n\xi) \Vert_2 = 0, \; \text{for all } x \in M, \, \xi \in H^\bot.$$
Then any $B$-$B$ sub-bimodule $K$ of $L^2(M)$ with $\dim(K_B) < \infty$ is contained in $H$. In particular, the quasi-normalizer $\mathcal{QN}_M(B)''$ is contained in $H \cap H^*$.
\end{lem}

Finally we state a specific intertwining lemma, more adapted to crossed-product von Neumann algebras. Assume that $\Gamma$ is a discrete countable group, and that $\mathcal S$ is a family of subgroups of $\Gamma$. Following \cite[Definition 15.1.1]{brownozawa}, we say that a subset $F$ of $\Gamma$ is small relative to $\mathcal S$, if it is of the form $\cup_{i=1}^n g_i\Sigma_i h_i$, for some $g_1,\cdots,g_n,h_1,\cdots,h_n \in \Gamma$, and $\Sigma_1,\cdots,\Sigma_n \in \mathcal S$.\\
Also, for any $F \subset \Gamma$, denote by $P_F \in B(L^2(\tilde M))$ the projection onto $\overline{\spa} \lbrace au_g \, | \, a \in \tilde A, g \in F \rbrace$.

\begin{lem}[Vaes, \cite{vaesonecoho}]
\label{intertwiningmixing}
Assume that $\Gamma \curvearrowright N$ is an action on a finite von Neumann algebra, and write $M = N \rtimes \Gamma$.
Let $p \in M$ be a projection and $B \subset pMp$ be a von Neumann subalgebra. The following are equivalent.
\begin{itemize}
\item $B \nprec_M N \rtimes \Sigma$, for every $\Sigma \in \mathcal{S}$ ;
\item There exists a net of unitaries $w_i \in \mathcal U(B)$ such that $\Vert P_F(w_i) \Vert_2 \rightarrow 0$ for every subset $F \subset \Gamma$ that is small relative to $\mathcal S$.
\end{itemize}
\end{lem}

\subsection{Gaussian actions}
We will use the following definition of the Gaussian functor, taken from \cite{vaesonecoho}. It can be checked that this characterizes both of the constructions given in \cite[Appendix A.7]{bhv} or \cite{ps}.\\

Assume that $H_\R$ is a real Hilbert space. Denote by $(A,\tau)$ the unique pair of an abelian von Neumann algebra $A$ with a trace $\tau$, and $A$ is generated by unitaries $(w(\xi))_{\xi \in H_\R}$ such that :
\begin{enumerate}[a)] 
\item $w(0) = 1$ and $w(\xi + \eta) = w(\xi)w(\eta)$, $w(\xi)^* = w(-\xi)$, for all $\xi,\eta \in H_\R$;
\item $\tau(w(\xi)) = exp(-\Vert \xi \Vert^2)$, for all $\xi \in H_\R$.
\end{enumerate}
It is easy to check that these conditions imply that the vectors $(w(\xi))_{\xi \in H_\R}$ are linearly independent and span a weakly dense $^*$-subalgebra of $A$, so that $(A,\tau)$ is indeed unique.

Now, for any orthogonal operator $U \in \mathcal{O}(H_\R)$, one can define a trace preserving automorphism $\theta_U$ of $A$ by the formula $\theta_U(w(\xi)) = w(U\xi)$.
Hence, to any orthogonal representation $\pi : \Gamma \rightarrow \mathcal{O}(H_\R)$ of a group $\Gamma$, one can associate a unique trace preserving action $\sigma_\pi$ of $\Gamma$ on $A$ such that $(\sigma_\pi)_g(w(\xi)) = w(\pi(g)\xi)$. This action $\sigma_\pi$ is called the \textit{Gaussian action} associated to $\pi$. In that context, $A$ will also be denoted $A_\pi$.\\

In the sequel, $\Gamma$ will denote a discrete countable group, and all the representations considered are assumed to be orthogonal.

\begin{rem}
Let $\pi$ a representation of $\Gamma$ and write $A = L^\infty(X,\mu)$. Naturally, $\sigma_\pi$ induces a measure preserving action of $\Gamma$ on $(X,\mu)$. Abusing with terminology, this action is also called the Gaussian action associated to $\pi$.
\end{rem}

\begin{exmp}
\label{bernoulli}
If $\Gamma$ acts on a countable set $I$, then the Gaussian action associated to the representation $\pi : \Gamma \rightarrow \mathcal{O}(\ell^2_\R(I))$ is the generalized Bernoulli action with diffuse basis $\Gamma \curvearrowright [0,1]^I$.
\end{exmp}
\begin{proof} Denote by $\mu_0$ the Gaussian probability measure on $\R$ : $$\mu_0 = \frac{1}{\sqrt{2\pi}}\exp(-x^2/2)dx,$$ and put $X = \R^I$, equipped with the product measure $\mu = \otimes_I \mu_0$. Also, for all $k \in I$ denote by $P_k : X \rightarrow \R$ the projection on the $k^{\text{th}}$ component. Then $(P_k)_{k \in I}$ is an orthonormal family in $L^2(X,\mu)$, so that one can define an embedding $\phi : \ell^2_\R(I) \rightarrow L^2_\R(X,\mu)$ by $\phi(\delta_k) = P_k$, for all $k \in I$.\\
Now, for all $\xi \in \ell^2_\R(I)$, put $w(\xi) = exp(i\sqrt{2}\phi(\xi)) \in \mathcal{U}(L^\infty(X,\mu))$. It is easily checked that these vectors satisfy conditions a) and b) above, and that the action of $\Gamma$ on $I$ is transformed into a shift of variables. Finally, the last thing to verify is that the von Neumann algebra $A$ generated by the $w(\xi)$'s is equal to $L^\infty(X,\mu)$. To do so, fix an increasing sequence $(X_n)_n$ of compacts subsets of $X$ such that $\cup_n X_n = X$ and $\lim_n \mu(X_n) = 1$ and put $p_n = \mathbf{1}_{X_n} \in L^\infty(X,\mu)$, $n \in \N$. Stone-Weiertrass' theorem implies that for all $n$, $Ap_n$ contains $C(X_n)$, showing that $A = L^\infty(X)$.
\end{proof}

\begin{lem}
\label{tensor product}
Let $\pi$ be a representation of $\Gamma$. Then $A_{\pi \oplus \pi} \simeq A_\pi \overline{\otimes} A_\pi$ and under this identification, $\sigma_{\pi \oplus \pi} = \sigma_\pi \otimes \sigma_\pi$.
\end{lem}
\begin{proof}
Note that $A_\pi \otimes A_\pi$ is generated by the unitary elements $w(\xi) \otimes w(\eta)$, for $\xi,\eta \in H_\R$, which satisfy the same relations as the $w(\xi \oplus \eta)$'s. Therefore the map $w(\xi \oplus \eta) \mapsto w(\xi) \otimes w(\eta)$, $\xi,\eta \in H_\R$ extends to a $^*$-isomorphism from $A_{\pi \oplus \pi}$ onto $A_\pi \otimes A_\pi$, that intertwines the actions $\sigma_{\pi \oplus \pi}$ and $\sigma_\pi \otimes \sigma_\pi$.
\end{proof}

Using an explicit construction of the Gaussian action (\emph{e.g.} \cite{ps}), one can see that for a representation $\pi$ of $\Gamma$, $L^2(A_\pi,\tau)$ is isomorphic (as a $\Gamma$-representation) to the symmetric Fock space $\mathcal{S}(H) = \C\Omega \oplus \bigoplus_{n \geq 1} H^{\odot n}$ of $H$. From that remark follows the following result (\cite{ps}).\\
Here $\sigma^0_\pi$ denotes the unitary representation of $\Gamma$ on $L^2(A_\pi,\tau)\ominus \C$ induced by $\sigma_\pi$.

\begin{prop}[Peterson-Sinclair, \cite{ps}]
\label{stableproperties}
Let $\pi$ a representation of $\Gamma$. Let $\mathcal P$ be any property in the following list :
\begin{enumerate}
\item being mixing ;
\item being mixing relative to a family $\mathcal S$ of subgroups of $\Gamma$ ;
\item being tempered.
\end{enumerate}
Then $\pi$ has property $\mathcal P$ if and only if $\sigma^0_\pi$ does.
\end{prop}

As pointed out by Sinclair (\cite{sinclair1}), the previous proposition is also valid for the property : ``having a tensor power which is tempered''.

As promised in the introduction, we end this section by showing, for a large class of groups the existence of Gaussian actions satisfying the assumptions of Theorems A and B, but which are not Bernoulli actions.

\begin{prop}
\label{nonbernoulli1}
If $\pi$ is a strongly $\ell^p$ representation, $p > 2$ which is not tempered, the associated Gaussian action is not a Bernoulli action (with diffuse basis). 
\end{prop}
\begin{proof}
Assume that two representations $\pi$ and $\rho$ induce conjugate Gaussian actions. Then Proposition \ref{stableproperties} implies that $\pi$ is mixing (\emph{resp.} tempered) if and only if $\rho$ is mixing (\emph{resp.} tempered). But for a representation $\Gamma \rightarrow \mathcal O(\ell^2(I))$ coming from an action $\Gamma \curvearrowright I$, being mixing implies being tempered (because the stabilizers have to be finite).\\
Therefore if a representation is mixing but not tempered, its Gaussian action cannot be conjugate to a generalized Bernoulli action with diffuse basis.
\end{proof}

\begin{prop}[Bekka]
\label{nonbernoulli2}
Every lattice $\Gamma$ in a non-compact, simple Lie group $G$ with finite center admits a unitary representation which is strongly $\ell^p$ for some $p > 2$, but not tempered.
\end{prop}
\begin{proof}
It is a known fact that $G$ admits an irreducible representation $\pi$ with no invariant vectors which is not strongly $\ell^q$, for some $q > 2$. By \cite{CHH}, $\pi$ is not weakly contained in the regular representation of $G$. But by Th\'eor\`eme 2.4.2 and Th\'eor\`eme 2.5.2 in \cite{Cowling}, there exist a $p > 2$ such that $\pi$ is strongly $\ell^p$.\\
We check that $\pi_{|\Gamma}$ satisfies the proposition.
It is easy to check that being strongly $\ell^p$ is stable by restriction to a lattice, so we are left to prove that $\pi_{|\Gamma}$ is not weakly contained in the left regular representation $\lambda_\Gamma$ of $\Gamma$. Denote by $\lambda_G$ the left regular representation of $G$.\\
Assume by contradiction that $\pi_{|\Gamma}$ is weakly contained in $\lambda_\Gamma$. Then by stability of weak containment under induction, we get that $\Ind_\Gamma^G(\pi_{|\Gamma})$ is weakly contained in $\lambda_G = \Ind_\Gamma^G(\lambda_\Gamma)$. However, $\Ind_\Gamma^G(\pi_{|\Gamma}) = \pi \otimes \Ind_\Gamma^G(1_\Gamma)$, and since $\Gamma$ has finite co-volume in $G$, the trivial $G$-representation is contained in $\Ind_\Gamma^G(1_\Gamma) = \lambda_{G/\Gamma}$. Altogether, we get that $\pi$ is weakly contained in $\lambda_G$, which is absurd.
\end{proof}

\begin{rem}
Every ICC lattice $\Gamma$ in $Sp(n,1)$ admits a strongly $\ell^p$ representation such that the crossed-product von Neumann algebra of the associated Gaussian action is not isomorphic to the crossed-product algebra of a Bernoulli action with diffuse basis.\\
{\it Indeed}, propositions \ref{nonbernoulli1} and \ref{nonbernoulli2} provide a strongly $\ell^p$ ($p > 2$) representation $\pi$ such that the associated Gaussian action $\sigma$ is not conjugate to a Bernoulli action (with diffuse basis). But Theorem 0.3 in \cite{popasup07} applies, so that $\sigma$ is OE-superrigid. Indeed, $\Gamma$ has property (T) and is ICC, $\sigma$ is free and mixing because $\pi$ is mixing (hence faithful since $\Gamma$ is ICC), and the next section shows that Gaussian actions are $s$-malleable in the sense of Popa.
Moreover, \cite{PV12} implies that since $\Gamma$ is hyperbolic, the crossed-product von Neumann algebra associated to $\sigma$ admits a unique Cartan subalgebra up to unitary conjugacy. By \cite{FMII}, we obtain that $\sigma$ is W$^*$-superrigid.
\end{rem}

\section{A localisation theorem for rigid subalgebras in the crossed-product}

The goal of this section is to prove Theorem \ref{adapted IPV}, and Theorem \ref{ultraproduct} allowing to locate rigid subalgebras in the crossed-product von Neumann $M$ associated to a Gaussian action, or in its ultraproduct $M^\omega$.

\subsection{The malleable deformation associated to a Gaussian action}
\label{notationsutiles}

From now on, $\pi : \Gamma \rightarrow \mathcal{O}(H_\R)$ will denote a fixed orthogonal representation of a countable discrete group $\Gamma$ on a separable real Hilbert space. In this fixed situation, we will remove all the $\pi$'s in the notations, and simply denote by $\sigma : \Gamma \curvearrowright A$ the Gaussian action associated to $\pi$. We use the standard $s$-malleable deformation of $\sigma$ (\cite{ps}). We recall the construction for convenience.\\

Consider the action $\sigma \otimes \sigma$ of $\Gamma$ on $A \overline{\otimes} A$. By lemma \ref{tensor product}, this is the Gaussian action associated to $\pi \oplus \pi$.\\
Define on $H_\R \oplus H_\R$ the operators $$\rho = \begin{pmatrix} 1&0\\0&-1 \end{pmatrix} \text{ and } \theta_t = \begin{pmatrix} \cos(\pi t/2) & -\sin(\pi t /2) \\ \sin(\pi t /2) & \cos(\pi t/2) \end{pmatrix}, \, t \in \mathbb{R}.$$
Here are some trivial facts about these operators :
\begin{itemize}
\item $\forall t \in \mathbb{R}, \, \rho \circ \theta_t = \theta_{-t} \circ \rho$ ;
\item $\theta_t$ and $\rho$ commute with $(\pi \oplus \pi)(g)$ for all $g \in \Gamma$, $t \in \mathbb{R}$ ;
\item $\forall s,t \in \mathbb{R}, \, \theta_s \circ \theta_t = \theta_{t+s}$.
\end{itemize}
Therefore $\rho$ and $(\theta_t)$ induce respectively an automorphism $\beta$ and a one-parameter family $(\alpha_t)$ of automorphisms of $A \ootimes A$ that commute with $\sigma \otimes \sigma$, and such that $ \beta \circ \alpha_t = \alpha_{-t} \circ \beta$ for all $t \in \mathbb{R}$. Observe also that $\alpha_1 = \varepsilon \circ \beta$, where $\varepsilon$ is the flip $a \otimes b \mapsto b \otimes a$.

Now consider the crossed-product von Neumann algebras $M = A \rtimes \Gamma$ and $\tilde{M} = (A \ootimes A) \rtimes_{\sigma \otimes \sigma} \Gamma$. View $M$ as a subalgebra of $\tilde{M}$ using the identification $M \simeq (A \ootimes 1)\rtimes \Gamma$.
The automorphisms defined above then extend to automorphisms of $\tilde{M}$ still denoted $(\alpha_t)$ and $\beta$, in a way such that $\alpha_t(u_g) = \beta(u_g) = u_g$, for all $g \in \Gamma$.

Being $s$-malleable, this deformation satisfies Popa's transversality property.

\begin{lem}[Popa's transversality argument, \cite{popasup}]
\label{transversality}
For any $x \in M$ and $t \in \R$ one has $$\Vert x - \alpha_{2t}(x) \Vert_2 \leq 2\Vert \alpha_t(x) - E_M \circ \alpha_t(x)\Vert_2.$$
\end{lem}

We then check in the following two lemmas that the inclusion $M \subset \tilde M$ satisfies the standard spectral gap property (see \cite{popasup}), which goes with rigidity phenomena.

\begin{lem}[Spectral Gap 1]
\label{Spectral Gap}
Let $M \subset \tilde{M}$ be finite von Neumann algebras and put $H = L^2(\tilde{M}) \ominus L^2(M)$, with the natural $M$-$M$ bimodule structure coming from $_ML^2(\tilde{M})_M$. Assume that some tensor power of $_MH_M$ is weakly contained in the coarse bimodule :
$$\exists K \geq 1, \: H^{\otimes_M K} := H \otimes_M \cdots \otimes_M H \subset_{w} L^2(M) \otimes L^2(M).$$
Let $\omega \in \beta \N \setminus \N$ be a free ultrafilter on $\N$. Then for every subalgebra $Q \subset M$ with no amenable direct summand, one has $Q' \cap \tilde{M}^\omega \subset M^\omega$.
\end{lem}
\begin{proof}
First, note that if $H^{\otimes_M K}$ is weakly contained in the coarse $M$-$M$ bimodule, then this is also the case of $H^{\otimes_M K+1}$. Hence one can assume that $K$ is of the form $K = 2^k$, which will be used later.\\
Now fix $Q \subset M$ such that $Q' \cap \tilde{M}^\omega \nsubseteq M^\omega$. We will show that $Q$ has an amenable direct summand.\\
Since $Q' \cap \tilde{M}^\omega \nsubseteq M^\omega$, there exist a sequence $x_n \in (\tilde{M})_1$ such that :
\begin{itemize}
\item $x_n \in L^2(\tilde{M}) \ominus L^2(M)$, for all $n \in \N$ ;
\item There exists $\varepsilon > 0$ such that $\Vert x_n \Vert_2 \geq \epsilon$ for  all $n \in \N$ ;
\item $\lim_n \Vert [u,x_n] \Vert_2 = 0$ for all $u \in \mathcal{U}(Q)$ ;
\item $x_n = x_n^*$.
\end{itemize}
Since $x_n \in (\tilde{M})_1$ for all $n \in \N$, the vectors $x_n \in H$ are left and right uniformly bounded, and one can consider the sequence $\xi_n = x_n \otimes_M \cdots \otimes_M x_n \in H^{\otimes_M K}$.\\
One checks that these are almost $Q$-central vectors, because the $x_n$'s are. Let's show that up to some slight modifications they are $Qq$-tracial as well, for some $q \in \mathcal{Z}(Q)$.\\

For all $n$, define by induction elements $y^{(n)}_i\in M, \, i= 1, \cdots, K$ by $y_1^{(n)} = E_M(x_n^2)$, $y_{i+1}^{(n)} = E_M(x_ny_i^{(n)}x_n)$.
Then an easy computation gives, for all $n \in \N$ and $a \in M$,
$$\langle a \xi_n , \xi_n \rangle = \langle a x_ny_{K-1}^{(n)}, x_n\rangle = \tau(ay_K^{(n)}).$$
Moreover, for all $n \in \N$, $\Vert x_n \Vert \leq 1$ implies $\Vert y_K^{(n)} \Vert \leq 1$. So taking a subsequence if necessary, one can assume that $(y_K^{(n)})$ converges weakly to some $b \in Q' \cap M_+$.\\
{\it Claim.} $\tau(b) \geq \varepsilon^{2K}$, so that $b \in M$ is a nonzero element.\\
To prove this claim, first observe that for any $1 \leq i,j \leq K-1$, one has :
\begin{align*}
\tau(y_i^{(n)}y_{j+1}^{(n)}) & = \tau(y_i^{(n)}E_M(x_ny_j^{(n)}x_n)) = \tau(y_i^{(n)}x_ny_j^{(n)}x_n)\\
& = \tau(E_M(x_ny_i^{(n)}x_n)y_j^{(n)}) = \tau(y_{i+1}^{(n)}y_j^{(n)}).
\end{align*}
Remembering that $K = 2^k$, the relation above and Cauchy-Schwarz inequality give :
\begin{align*}
\tau(y_K^{(n)}) = \tau(y_{2^k}^{(n)}) & =  \tau(y_{2^{k-1}}^{(n)}y_{2^{k-1}}^{(n)})\\
& \geq \tau(y_{2^{k-1}}^{(n)})^2 \geq \cdots \geq \tau(y_1^{(n)})^{2^(k-1)}\\
& = \tau(x_n^2)^{K/2} \geq \varepsilon^{K}.
\end{align*}
This proves the claim. Therefore there exists $\delta>0$ such that $q = \chi_{[\delta, \infty[}(E_Q(b)) \neq 0$. Note that $q \in \mathcal{Z}(Q)$ and take $c \in \mathcal{Z}(Q)_+$ such that $q = cE_Q(b)$.\\
Finally, we get that the sequence $\eta_n = c^{1/2} \cdot \xi_n \in H^{\otimes_M K}$ satisfies :
\begin{itemize}
\item $(\eta_n)$ is almost $Qq$-tracial : $\forall a \in Qq$, $\lim_n \langle a \eta_n , \eta_n \rangle = \tau(c^{1/2}ac^{1/2}b) = \tau(aq)$.
\item $(\eta_n)$ is almost $Q$-central.
\end{itemize}
Therefore as $Qq$-$Qq$ bimodules, we have : $$L^2(Qq) \subset_w H^{\otimes_M K} \subset_w L^2(M) \otimes L^2(M) \subset_w L^2(Qq) \otimes L^2(Qq),$$
so that $Qq$ is amenable.
\end{proof}

\begin{lem}[Spectral gap 2]
\label{weaklycontained}
Assume that the representation $\pi$ is such that $\pi^{\otimes K} \prec \lambda$ for some $K \geq 1$, then the bimodule $_MH_M = L^2(\tilde{M}) \ominus L^2(M)$ is such that $H^{\otimes_M \, K}$ is weakly contained in the coarse bimodule.
\end{lem}
\begin{proof}
As in the proof of Lemma 3.5 in \cite{vaesonecoho}, for any representation $\eta : \Gamma \rightarrow \mathcal{U}(K)$, define a $M$-$M$ bimodule structure $\mathcal H^\eta$ on the Hilbert space $L^2(M) \otimes K$, by $$(au_g) \cdot (x \otimes \xi) \cdot (bu_h) = au_gxbu_h \otimes \eta_g(\xi), \; \text{for all } a,b \in A,\, g,h \in \Gamma, \, x\in M, \, \xi \in K.$$
Since $A$ is amenable then $\mathcal H^\eta$ is weakly contained in the coarse bimodule whenever $\eta$ is tempered.\\
But remark that the $M$-$M$ bimodule $L^2(\tilde{M}) \ominus L^2(M)$ is isomorphic to $\mathcal H^{\sigma_\pi^0}$, and that for two representation $\eta_1$, $\eta_2$ of $\Gamma$, $\mathcal H^{\eta_1} \otimes_M \mathcal H^{\eta_2} = \mathcal H^{\eta_1 \otimes \eta_2}$. Moreover, from the comment after proposition \ref{stableproperties}, we have that $(\sigma^0_\pi)^{\otimes K}$ is tempered. This ends the proof.
\end{proof}

\subsection{Position of rigid subalgebras in $M$}

Our aim here is to show the following theorem, which is an adaptation of \cite[Theorem 4.2]{ipv} in the framework of Gaussian actions.

\begin{thm}
\label{adapted IPV}
Assume that $\pi$ is mixing relative to a family $\mathcal S$ of subgroups of $\Gamma$. Put $M = A \rtimes \Gamma$ and define $\tilde{M}$ and $(\alpha_t)$ as in the previous subsection.\\
Let $Q \subset pMp$ be a subalgebra such that there exist $z \in \tilde{M}$, $t_0 = 1/2^n$ ($n \geq 0$) and $c > 0$ satisfying
$$\vert \tau(\alpha_{t_0}(u^*)zu)\vert \geq c, \; \text{for all } u \in \mathcal U(Q).$$
Put $P = \mathcal{N}_{pMp}(Q)''$. Then at least one of the following assertions occurs.
\begin{enumerate}
\item $P \prec_M A \rtimes \Sigma$, for some $\Sigma \in \mathcal{S}$ ;
\item $Q \prec_M L\Gamma$. 
\end{enumerate}
Moreover, if the elements of $\mathcal S$ are almost malnormal in the sense of definition \ref{malnormal}, then the above dichotomy can be replaced by :
\begin{enumerate}
\item[(1')] $Q \prec_M \C1$ ;
\item[(2')] $P \prec_M A \rtimes \Sigma$, for some $\Sigma \in \mathcal{S}$ ;
\item[(3')] $P \prec_M L\Gamma$.
\end{enumerate}
\end{thm}

Before proving this theorem, we give two lemmas regarding the position of normalizers of subalgebras in $M$ in some specific situations.
The first lemma below is Lemma 3.8 in \cite{vaesonecoho}, whereas Lemma \ref{norm2} is a variation of Lemma 4.2 in \cite{vaesbimodule}.

\begin{lem}[Vaes, \cite{vaesonecoho}]
\label{norm1}
Assume that $\pi$ is mixing relative to a family $\mathcal S$ of subgroups of $\Gamma$. Let $N$ be a finite von Neumann algebra, and $\Gamma \curvearrowright N$ any action. Put $M_0 = N \rtimes \Gamma$, and $\tilde{M}_0 = (A \otimes N) \rtimes \Gamma$.\\
Let $p \in M_0$ be a projection and $Q \subset pM_0p$ a von Neumann subalgebra such that $Q \nprec_{M_0} N \rtimes \Sigma$, for all $\Sigma \in \mathcal{S}$. Then $\mathcal{N}_{p\tilde{M}_0p}(Q) \subset pM_0p$.
\end{lem}

The previous lemma will be used in the special cases where $\Gamma \curvearrowright N$ is either the Gaussian action (and $M_0 = M$, $\tilde{M}_0 = \tilde{M}$), or the trivial action (and $M_0 = L\Gamma$, $\tilde{M}_0 = M$).

\begin{lem}
\label{norm2}
Assume that $\Sigma < \Gamma$ is a subgroup. Put $I = \Gamma / \Sigma$ and consider the action of $\Gamma$ on $I$ obtained by multiplication to the left. For a subset $I_1 \subset I$, write $\Stab(I_1) = \lbrace g \in \Gamma \, | \, g\cdot i = i, \forall i \in I_1 \rbrace$ and $\Norm(I_1) = \lbrace g \in \Gamma \, | \, g\cdot I_1 = I_1 \rbrace$. Let $\Gamma \curvearrowright N$ be any  action on a finite von Neumann algebra $N$, and set $M_0 = N \rtimes \Gamma$.\\
Let $p \in L(\Stab(I_1))$ be a projection and $B \subset pL(\Stab(I_1))p$ a subalgebra such that for all $i \in I \setminus I_1$, $B \nprec_{\Stab(I_1)} L(\Stab(I_1 \cup \lbrace i \rbrace))$. Then the quasi-normalizer of $B$ in $pM_0p$ is contained in $p\left(N \rtimes \Norm(I_1)\right)p$.
\end{lem}
The proof of this lemma is exactly the same as the proof of 4.2 in \cite{vaesbimodule}. We include it for the sake of completeness.
\begin{proof}
By the comment following Theorem \ref{intertwining} (diagonal version of that theorem), there exists a sequence $w_n \in \mathcal{U}(B)$ such that for all $i \in I \setminus I_1$, and all $g,h \in Stab(I_1)$, $$\lim_n \Vert E_{L(\Stab(I_1 \cup \lbrace i \rbrace)}(u_gw_nu_h) \Vert_2 = 0.$$
Define a $B$-$B$ bimodule $H \subset L^2(M_0)$ as the closed span of the $xu_g$, with $x \in N$ and $g \in \Gamma$ such that $gI_1 \subset I_1$. Observe that $H \cap H^* = L^2(N \rtimes \Norm(I_1))$. Hence by the relative intertwining lemma \ref{relativeintertwining}, it is enough to show that for all $x \in M_0$, and all $\xi \in H^\perp$, $$\lim_n \Vert e_B(xw_n\xi) \Vert_2 = 0.$$
We can assume $x = au_g$, $\xi = bu_h$, for $a,b \in N$, $g,h \in \Gamma$, and $hI_1 \nsubseteq I_1$. So write $w_n = \sum_{k \in \Stab(I_1)} \lambda_{n,k}u_k$, $\lambda_{n,k} \in \C$. We have 
$$\Vert E_{L(\Stab(I_1))}(au_gw_nbu_h)\Vert_2^2 = \sum_{k \in \Stab(I_1) \cap g^{-1}\Stab(I_1)h^{-1}} \vert \tau(a \sigma_{gk}(b) \vert^2 \vert \lambda_{n,k} \vert^2.$$
But note that if $\mathcal I = \Stab(I_1) \cap g^{-1}\Stab(I_1)h^{-1}$ is non-empty, then it is contained in $k_0 \Stab(I_1 \cup \lbrace i \rbrace)$, for any $k_0 \in \mathcal{I}$, $i \in hI_1 \setminus I_1$. Hence,
$$\Vert E_{L(\Stab(I_1))}(au_gw_nbu_h)\Vert_2 \leq \Vert a \Vert_2 \Vert b \Vert_2 \Vert E_{L(\Stab(I_1 \cup \lbrace i \rbrace)}(u_{k_0^{-1}}w_n)\Vert_2,$$
which goes to $0$ as $n$ goes to infinity because $k_0 \in \mathcal I \subset \Stab(I_1)$. Since $B \subset L(\Stab(I_1))$, we get the result.
\end{proof}

Lemma \ref{norm2} will be used by the means of the following proposition.

\begin{cor}
\label{norm3}
Let $\Gamma \curvearrowright N$ be any action, and put $M_0 = N \rtimes \Gamma$.\\
If $Q \subset pM_0p$ is a diffuse von Neumann algebra such that $Q \prec_{M_0} L\Sigma$ for an almost malnormal subgroup $\Sigma \subset \Gamma$ (Definition \ref{malnormal}), then $P \prec_{M_0} N \rtimes \Sigma$, where $P = \mathcal{QN}_{pM_0p}(Q)''$.
\end{cor}
\begin{proof}
Use the notations of Lemma \ref{norm2} and take $n \geq 1$ such that $\Sigma$ is $n$-almost malnormal. Since $Q$ is diffuse, one has $Q \nprec_{M_0} L(\Stab(I_0))$ for $\vert I_0 \vert \geq n$ because $\Stab(I_0) = \cap_{g\Sigma \in I_0} g\Sigma g^{-1}$ is finite.
Hence one can consider a maximal finite subset $I_1 \subset I = \Gamma / \Sigma$ such that $\Sigma \in I_1$ and $Q \prec_{M_0} L(\Stab(I_1))$.\\
But as explained in \cite[Remark 3.8]{vaesbimodule}, there exist projections $q_0 \in Q$, $p_0 \in L(\Stab(I_1))$, a $*$-homomorphism $\psi : q_0Qq_0 \rightarrow p_0L(\Stab(I_1))p_0$, and a non-zero partial isometry $v \in q_0M_0p_0$ such that $xv = v\psi(x)$, for all $x \in q_0Qq_0$, and $$\psi(q_0Qq_0) \nprec_{L(\Stab(I_1))} L(\Stab(I_1 \cup \lbrace i \rbrace)),$$
for all $i \in I \setminus I_1$.\\
By Lemma \ref{norm2}, this implies that $v^* Pv \subset \mathcal{QN}_{p_0M_0p_0}(\psi(q_0Qq_0))'' \subset N \rtimes \Norm(I_1)$. Therefore $P \prec_{M_0} N \rtimes \Stab(I_1)$, because $\Stab(I_1) < \Norm(I_1)$ is a finite index subgroup.
But by assumption $\Sigma \in I_1$, so that $\Stab(I_1) \subset \Sigma$.
\end{proof}

\begin{proof}[Proof of Theorem \ref{adapted IPV}]
To simplify notations, we assume that $p = 1$ ; the proof is exactly the same in the general case.\\
Suppose that (1) is false, that is, no corner of $P$ embeds into $A \rtimes \Sigma$ inside $M$, for all $\Sigma \in \mathcal S$. We will prove that $Q \prec_M L\Gamma$.\\
First, a classical convex hull trick (as in the proof of Lemma 5.2 in \cite{popasup}) implies that there exist a non-zero partial isometry $v_0 \in \tilde{M}$ such that for all $x \in Q$, $xv_0 = v_0 \alpha_{t_0}(x)$. In particular, $v_0$ is $Q$-$\alpha_{t_0}(Q)$ finite.\\
Now we show that there exists a non-zero element $a \in \tilde{M}$ which is $Q$-$\alpha_1(Q)$ finite. To do so, observe that if $v \in \tilde{M}$ is $Q$-$\alpha_t(Q)$ finite for some $t > 0$, then for any $d \in \mathcal{N}_M(Q)$, the element $\alpha_t(\beta(v^*)dv)$ is $Q$-$\alpha_{2t}(Q)$ finite. The following claim is enough to prove the existence of $a$.\\
\textit{Claim.} For any nonzero element $v \in \tilde{M}$, there exists $d \in \mathcal{N}_M(Q)$ such that $\beta(v^*)dv \neq 0$.\\
Assume by contradiction that there exists a $v \neq 0$ with $\beta(v^*)dv = 0$, for all $d \in \mathcal{N}_M(Q)$. Denote by $q \in \tilde{M}$ the projection onto the closed linear span of $\lbrace range(dv)\, | \, d \in \mathcal{N}_M(Q) \rbrace$. We see that $\beta(q)q=0$ and $q \in P' \cap \tilde M$. By lemma \ref{norm1}, since $P \nprec_M A \rtimes \Sigma$ for all $\Sigma \in \mathcal S$, we have $P' \cap \tilde M \subset M$, so that $q \in M$ and $\beta(q) = q$. Hence $q = 0$, which contradicts the fact that $q \geq vv^* \neq 0$.\\

Considering the $Q$-$\alpha_1(Q)$ bimodule $\overline{\spa(Qa\alpha_1(Q))}$, we see that $\alpha_1(Q) \prec_{\tilde M} Q$. Let's check that this implies that $Q \prec_M L\Gamma$.\\
Reasoning again by contradiction, assume that $Q \nprec_M L\Gamma$. Popa's intertwining lemma \ref{intertwining} then implies that there exists a sequence $(w_n) \subset \mathcal{U}(Q)$ such that for all $x,y \in M$, $\lim_n \Vert E_{L\Gamma}(xw_ny) \Vert_2 = 0$.\\
We claim that $\lim_n \Vert E_M(x \alpha_1(w_n) y) \Vert_2 = 0$ for all $x,y \in \tilde{M}$.
By a linearity/density argument, it suffices to prove this equality for $x = (a \otimes b)u_g$ and $y = (c \otimes d)u_h$, with $a,b,c,d \in A$, $g,h \in \Gamma$. Now writing $w_n = \sum_{k \in \Gamma} x_{k,n}u_k$, an easy calculation gives
\begin{align*}
\Vert E_M(x \alpha_1(w_n) y) \Vert^2_2 \, & = \,\left\Vert E_M \left( \sum_k \left(a\sigma_{gk}(c) \otimes b\sigma_g(x_{k,n})\sigma_{gk}(d)\right) u_{gkh} \right) \right\Vert^2_2\\
& = \, \sum_k \Vert a \sigma_{gk}(c) \Vert^2_2 \, \vert \tau(b \sigma_g(x_{k,n}) \sigma_{gk}(d)) \vert^2\\
& \leq \, \sum_k \Vert a \Vert^2_{\infty} \Vert c \Vert^2_2  \, \vert \tau(b\sigma_g(x_{k,n}) \sigma_{gk}(d)) \vert^2\\
& = \, \Vert a \Vert^2_{\infty} \Vert c \Vert^2_2 \Vert E_{L\Gamma}\left( (bu_g) w_n d \right) \Vert^2_2,
\end{align*}
which tends to $0$ when $n$ goes to infinity. This contradicts $\alpha_1(Q) \prec_{\tilde{M}} M$.\\

For the moreover part, assume that \textit{(1')} and \textit{(2')} are not satisfied. By Proposition \ref{norm3}, since \textit{(2')} does not hold true, we get that $Q \nprec_M L\Sigma$ for all $\Sigma \in \mathcal{S}$. Furthermore, the first part of the theorem implies that $Q \prec_M L\Gamma$. Now, proceeding as in the proof of Theorem 4.1 (step 5) in \cite{popamal1}, one checks that Lemma \ref{norm1} implies the result.
\end{proof}

\subsection{Position of rigid subalgebras in $M^\omega$}

In view of studying property Gamma, our goal is now the following theorem, that should be compared to Theorem 3.2 of \cite{ioana5}.

\begin{thm}
\label{ultraproduct}
Assume that $\pi$ is mixing relative to a finite family $\mathcal S$ of almost malnormal subgroups of $\Gamma$. Let $\omega \in \beta \N \setminus \N$ and let $B \subset M^\omega$ be a von Neumann subalgebra such that the deformation converges uniformly to the identity on $(B)_1$. Then one of the following holds.
\begin{enumerate}
\item $B \prec_{M^\omega} \C1$ ;
\item $B \subset A^\omega \rtimes \Gamma$ ;
\item $B' \cap M \prec_M A \rtimes \Sigma$, for some $\Sigma \in \mathcal{S}$ ;
\item $B' \cap M \prec_M L\Gamma$.
\end{enumerate}
\end{thm}

Before proving the theorem we recall some terminology and give a technical lemma, which is the first part of Lemma 3.8 in \cite{vaesonecoho}.\\
A subset $F$ of $\Gamma$ is said to be small relative to $\mathcal S$ if it is of the form $\cup_{i=1}^n g_i\Sigma_i h_i$, for some $g_1,\cdots,g_n,h_1,\cdots,h_n \in \Gamma$, and $\Sigma_1,\cdots,\Sigma_n \in \mathcal S$. We denote by $\mathcal S_s$ the set of all such small sets. For any $F \subset \Gamma$, denote by $P_F \in B(L^2(\tilde M))$ the projection onto $\overline{\spa} \lbrace au_g \, | \, a \in \tilde A, g \in F \rbrace$.\\
As observed in the proof of Lemma 5.5 in \cite{PV10}, though $P_F$ might not restrict to a bounded map on $\tilde M$ (for the norm $\Vert \cdot \Vert$) for any $F$, it restricts well to a completely bounded map whenever $F$ is a finite union of $g_i\Sigma_i h_i$'s with $g_i, h_i \in \Gamma$, and each $\Sigma_i < \Gamma$ being a subgroup. Indeed, if $F = g\Sigma h$, then $P_F(x) = u_gE_{\tilde A \rtimes \Sigma}(u_g^*xu_h^*)u_h$ is completely bounded on $\tilde M$. Now if $F = \cup_{i=1}^n F_i$, with $F_i = g_i\Sigma_i h_i$, then the projections $P_{F_i}$ commute and $P_F = 1 - (1-P_{F_1})\cdots (1 - P_{F_N})$ is completely bounded as well.

\begin{lem}[Vaes, \cite{vaesonecoho}]
\label{technicallemma}
Assume that $\pi$ is mixing relative to a family $\mathcal S$ of subgroups of $\Gamma$.\\
For a finite dimensional subspace $K \subset A \ominus \C1$, denote by $Q_K$ the orthogonal projection of $L^2(\tilde M)$ onto the closed linear span of $(a \otimes b)u_g$, $g \in \Gamma$, $a \in A$, $b \in K$.\\
For every finite dimensional $K \subset A \ominus \C1$, every $x \in (\tilde M)_1$ and every $\varepsilon > 0$, there exists $F \in \mathcal S_s$ such that $$\Vert Q_K(vx) \Vert_2 \leq \Vert P_F(v) \Vert_2 + \varepsilon, \; \text{for all } v \in (M)_1.$$
\end{lem}

\begin{proof}[Proof of Theorem \ref{ultraproduct}]
The proof goes in two steps. In the first step, we show that the result is true if we replace condition (2) by \[(2') \; \forall \varepsilon > 0, \, \forall x = (x_n) \in B, \, \exists F \in \mathcal S_s \, : \lim_{n \rightarrow \omega} \Vert P_F(x_n) - x_n \Vert_2 < \varepsilon.\]
The second step consists in showing that (2') implies (2) or (3).\\

{\sc Step 1.} Assume that (1) and (2') are not satisfied. We will show that there exist $t_0 =1/2^{n_0}$, $c > 0$ and $z_0 \in \tilde{M}$ such that \[\tag{3.a} \vert \tau (\alpha_t(u)z_0u^*) \vert \geq c, \text{ for all } u \in \mathcal U(B' \cap M).\]
Theorem \ref{adapted IPV} will then conclude.\\
Denote by $E \subset L^2(\tilde M^\omega)$ the $\Vert \cdot \Vert_2$-closed span of $\lbrace (P_F(x_n)) \, \vert \, (x_n) \in \tilde M^\omega, \, F \in \mathcal S_s \rbrace$, and by $P \in B(L^2(\tilde M^\omega))$ the orthogonal projection onto $E$. One checks that $P$ commutes with $\alpha_t$ for all $t \in \R$, and also with left and right actions of $M$. Moreover, $P(L^2(M^\omega)) \subset L^2(M^\omega)$.\\
Condition (2') being not satisfied, there exists $x \in B$, with $\Vert x \Vert_2 = 1$ such that $x \notin E$. Then $x - P(x) \in L^2(M^\omega)$ is non zero, and has a norm $\Vert \cdot \Vert_2$ smaller or equal to $1$. Fix $\varepsilon$ very small, say $\varepsilon = \Vert x - P(x) \Vert_2^2 / 1000 \leq 1/1000$, and take $y = (y_n) \in M^\omega$ such that $\Vert x - P(x) - y \Vert_2 \leq \varepsilon$. Also choose $t = 1/2^n$ such that $\Vert \alpha_t(x) - x \Vert_2 \leq \varepsilon$.\\
Then $y$ is easily seen to satisfy the following three conditions :
\begin{itemize}
\item $\Vert \alpha_t(y) - y \Vert_2 \leq 3\varepsilon$ ;
\item $\Vert [y,a] \Vert_2 \leq 2\varepsilon$, for all $a \in (B' \cap M)_1$ ;
\item $\lim_{n \rightarrow \omega} \Vert P_F(y_n) \Vert_2 \leq \varepsilon$, for all $F \in \mathcal S_s$.
\end{itemize}
We show that $t_0 = 2t$ and $z_0 = E_{\tilde M}(y{y}^*)$ satisfy (3.a).\\
Take $u \in \mathcal U(B' \cap M)$. For all $a \in M$ define $\delta_t(a) = \alpha_t(a) - E_M \circ \alpha_t(a) \in \tilde M \ominus M$.
Now, consider a finite dimensional subspace $K \subset A \ominus \C1$ such that $\Vert Q_K(\delta_t(u))-\delta_t(u) \Vert_2 < \varepsilon/ \lim_n \Vert y_n \Vert^2$, where $Q_K$ is defined as in Lemma \ref{technicallemma}. Note that $Q_K$ is right $M$-modular, and that $Q_K \circ E_M = 0$. We have
\begin{align*}
\Vert \delta_t(u)y \Vert_2^2 & = \lim_{n \rightarrow \omega} \langle \delta_t(u)y_ny_n^*,\delta_t(u) \rangle\\
& \approx_\varepsilon \lim_{n \rightarrow \omega} \langle \delta_t(u)y_ny_n^*,Q_K(\delta_t(u)) \rangle \\
& = \lim_{n \rightarrow \omega} \Vert Q_K(\delta_t(u)y_n) \Vert_2^2\\
& = \lim_{n \rightarrow \omega} \Vert Q_K(\alpha_t(u)y_n) \Vert_2^2\\
& \approx_{8\varepsilon} \lim_{n \rightarrow \omega} \Vert Q_K(y_n\alpha_t(u)) \Vert_2^2.
\end{align*}
But by Lemma \ref{technicallemma}, there exists $F \in \mathcal S_s$ such that for all $n$, \[\Vert Q_K(y_n\alpha_t(u)) \Vert_2^2 \leq \Vert P_F(y_n) \Vert_2^2 + \varepsilon.\]
Combining all these approximations and inequalities, we get on the first hand :
\[\tag{3.b} \Vert \delta_t(u)y \Vert_2^2 \leq \lim_{n \rightarrow \omega} \Vert P_F(y_n) \Vert_2^2 + 10\varepsilon \leq 11 \varepsilon.\]
On the other hand, Popa's transversality lemma implies $\Vert \alpha_{2t}(uy) - uy \Vert_2  \leq 2\Vert \delta_t(uy) \Vert_2$.
Since $\alpha_{2t}(u)y - uy \approx_{6\varepsilon} \alpha_{2t}(uy) - uy$, and $\delta_t(uy) \approx_{3\varepsilon} \delta_t(u)y$ (for the norm $\Vert \cdot \Vert_2$), we get
\[\Vert \alpha_{2t}(u)y - uy \Vert_2 \leq 2\Vert \delta_t(u)y \Vert_2 + 12\varepsilon.\]
Hence \[\Vert \alpha_{2t}(u)y - uy \Vert_2^2 \leq 4\Vert \delta_t(u)y \Vert_2^2 + 48\varepsilon\Vert \delta_t(u)y \Vert_2 + 144\varepsilon^2.\] 
But remember that $\Vert x - P(x) \Vert_2 \leq 1$ and $\varepsilon \leq 1/1000$, so $\Vert y \Vert_2 \leq 2$ and $144\varepsilon^2 \leq \varepsilon$.
We obtain using (3.b)
\[\Vert \alpha_{2t}(u)y - uy \Vert_2^2 \leq 4\Vert \delta_t(u)y \Vert_2^2 + 200\varepsilon \leq 300\varepsilon.\]
Developing $\Vert \alpha_{2t}(u)y - uy \Vert_2^2$, this implies 
\begin{align*}
\tau(\alpha_{t_0}(u)E_{\tilde M}(yy^*)u^*) = \tau(\alpha_{2t}(u)yy^*u^*)
& \geq \Vert y \Vert_2^2 - 150\varepsilon\\
& \geq (\Vert x - P(x)\Vert_2 - \varepsilon)^2 - 150\varepsilon\\
& \geq \Vert x - P(x)\Vert_2^2 - 152\varepsilon > 0,
\end{align*}
as desired.\\

{\sc Step 2.} Assume that condition (2') holds true, but conditions (2) and (3) are not satisfied. We will derive a contradiction. What follows should be compared to the proofs of Lemma \ref{norm2} and Proposition \ref{norm3}.\\
Consider an element $x \in B \setminus (A^\omega \rtimes \Gamma)$ and put $y = (y_n) = x - E_{A^\omega \rtimes \Gamma}(x)$. 
Remark that any element in $B' \cap M$ commutes with $y$, because $M \subset A^\omega \rtimes \Gamma$.\\
By condition (2'), there exist $\Sigma \in \mathcal S$ and $g,h \in \Gamma$ such that $$\lim_{n \rightarrow \omega} \Vert P_\Sigma(u_gy_nu_h) \Vert_2 \neq 0.$$

Now put $I = \bigsqcup_{\Sigma \in \mathcal S} \Gamma / \Sigma$. Since $\mathcal S$ is finite and its elements are malnormal subgroups in $\Gamma$, there exists a constant $\kappa > 0$ such that $\Stab(I_0)$ is finite for all $I_0 \subset I$ with $\vert I_0 \vert \geq \kappa$.\\
Since $y \perp A^\omega \rtimes \Gamma$ we get that $\lim_\omega \Vert P_{\Stab(I_0)}(u_{g'}y_nu_{h'}) \Vert_2 = 0$ for all $g',h' \in \Gamma$ whenever $\vert I_0 \vert \geq \kappa$.
Hence there exist $I_1 \subset I$ finite, $g_0,h_0 \in \Gamma$ such that 
\begin{itemize}
\item[(3.c)] $\lim_\omega \Vert P_{\Stab(I_1)}(u_{g_0}y_nu_{h_0}) \Vert_2 = c > 0$ ;
\item[(3.d)] $\lim_\omega \Vert P_{\Stab(I_1 \cup \lbrace i \rbrace)}(u_{g'}y_nu_{h'}) \Vert_2 = 0$, for all $i \notin I_1$, and $g',h' \in \Gamma$.
\end{itemize}
Put $\varepsilon = c/5$ and take $F \in \mathcal{S}_s$ such that $\Vert y - (P_F(y_n)) \Vert_2 < \varepsilon$. Recall that $P_F$ is completely bounded on $M$, so that $(P_F(y_n))$ is a bounded sequence in $M$.
Write $F = g_1\Sigma_1h_1 \cup \cdots \cup g_p\Sigma_p h_p$ as in the definition of small sets relative to $\mathcal S$, and set \[F_0 = \lbrace h \in \Gamma \, \vert \, \exists i \leq p \, : \, h^{-1}(h_i^{-1}\Sigma_i) \in I_1 \rbrace = \bigcup_{i=1}^p \bigcup_{g\Sigma_i \in I_1} h_i^{-1}\Sigma_ig^{-1}.\]
Thus $F_0$ is small relative to $\mathcal S$.\\
{\it Claim.} $\lim_\omega \Vert P_{\Stab(I_1)}(xP_F(y_n)\xi)\Vert_2 = 0$, for all $x \in M$, $\xi \in P_{\Gamma \setminus F_0}(L^2(M))$.\\
Since $P_F(y_n)$ is bounded in norm $\Vert \cdot \Vert$ (and by Kaplanski's theorem), it is sufficient to prove this claim for $x = u_g$, $g \in \Gamma$ and $\xi = u_h$, with $h \in \Gamma$ such that $h^{-1}(h_i^{-1}\Sigma_i) \notin I_1$ for all $1 \leq i \leq p$. Write $y_n = \sum_{k \in \Gamma} y_{n,k}u_k$ (and so $P_F(y_n) = \sum_{k \in F} y_{n,k}u_k$). We get
\begin{align*}
\Vert P_{\Stab(I_1)}(u_gP_F(y_n)u_h)\Vert_2^2 & = \Vert \sum_{k \in F \cap g^{-1}\Stab(I_1)h^{-1}} \sigma_g(y_{n,k})u_{ghk} \Vert_2^2 \\
& = \sum_{k \in gFh \cap \Stab(I_1)} \Vert y_{n,g^{-1}kh^{-1}} \Vert_2^2 \\
& \leq \sum_{i=1}^p \left( \sum_{k \in gg_i\Sigma_i h_i h \cap \Stab(I_1)} \Vert y_{n,g^{-1}kh^{-1}} \Vert_2^2 \right)
\end{align*}
Now, for $i \in \lbrace 1 \cdots p \rbrace$, if the set $F_i = gg_i\Sigma_i h_i h \cap \Stab(I_1)$ is non empty, then it is contained in $k_i\Stab(I_1 \cup \lbrace i_i \rbrace)$ for $k_i \in F_i$, $i_i = h^{-1}h_i^{-1}(\Sigma_i) = h^{-1}(h_i^{-1}\Sigma_i) \notin I_1$. Therefore,
\begin{align*}
\sum_{k \in F_i} \Vert y_{n,g^{-1}k h^{-1}} \Vert_2^2 \; & \leq \sum_{k \in \Stab(I_1 \cup \lbrace i_i \rbrace)} \Vert y_{n,g^{-1}k_i k h^{-1}} \Vert_2^2 \\
& = \Vert P_{\Stab(I_1 \cup \lbrace i_i \rbrace)}(u_{k_i^{-1}g}y_n u_h) \Vert_2^2,
\end{align*}
which goes to $0$, when $n \rightarrow \omega$, because of (3.d). So the claim is proven.\\
Since (3) is not satisfied, Lemma \ref{intertwiningmixing} implies that there exists a sequence of unitaries $v_k \in B' \cap M$ such that $\lim_k \Vert P_G(v_k) \Vert_2 = 0$ for all $G \in \mathcal{S}_s$.\\
Fix $k$ such that $\Vert P_{F_0}(v_ku_{h_0}) \Vert_2 < \frac{\varepsilon}{\sup_n \Vert P_F(y_n) \Vert}.$
We get for all $n$,
\begin{align*}
\Vert P_{\Stab(I_1)}(u_{g_0}v_k^*y_nv_ku_{h_0}) \Vert_2 & \leq  \Vert P_{\Stab(I_1)}(u_{g_0}v_k^*P_F(y_n)v_ku_{h_0}) \Vert_2 + \Vert y_n - P_F(y_n) \Vert_2\\
& \leq \Vert P_{\Stab(I_1)}(u_{g_0}v_k^*P_F(y_n)P_{\Gamma \setminus F_0}(v_ku_{h_0})) \Vert_2 + \varepsilon + \Vert y_n - P_F(y_n) \Vert_2.
\end{align*}
By the claim, we can choose $n$ large enough so that $\Vert P_{\Stab(I_1)}(u_{g_0}v_k^*y_nv_ku_{h_0}) \Vert_2 \leq 3\varepsilon$, and also $\Vert [ v_k , y_n ] \Vert_2 \leq \varepsilon$. Thus we obtain :
$$\lim_\omega \Vert P_{\Stab(I_1)}(u_{g_0}y_nu_{h_0}) \Vert_2 \leq 4\varepsilon < c,$$
which contradicts (3.c). The proof is complete.
\end{proof}

\begin{rem}
In fact, the above proof can be modified to handle the case where $\mathcal S$ is not necessarily finite, but satisfies the following condition : there exists some $n \geq 0$, such that for all $\Sigma_1,\cdots,\Sigma_n \in \mathcal S$, $g_1,\cdots,g_n \in \Gamma$ with $g_i\Sigma_i \neq g_j\Sigma_j$ (as subsets of $\Gamma$), the set $\cap_{i=1}^n g_i\Sigma_ig_i^{-1}$ is finite. Theorem B will remain true under this softer assumption (but still requiring that the elements of $\mathcal S$ are amenable).
\end{rem}

\section{Proof of the main results}

We start with a proposition combining the spectral gap argument with Theorem \ref{adapted IPV}.

\begin{prop}
\label{positioncommutant}
Assume that $\pi$ is mixing relative to a family $\mathcal S$ of amenable subgroups of $\Gamma$ and has some tensor power which is tempered. Then for any subalgebra $Q \subset pMp$ with no amenable direct summand, one has $$P := Q' \cap pMp \prec_M L\Gamma.$$
If in addition the elements of $\mathcal S$ are almost malnormal in $\Gamma$, then $$P \prec_M \C1 \text{ or } \mathcal N_{pMp}(P)'' \prec_M L\Gamma.$$
\end{prop}
\begin{proof}
Consider an amplification $Q^t \subset M$ with $t = 1/\tau(p)$, such that $Q = pQ^tp$.
Exactly as in the proof of Lemma 5.2 in \cite{popasup}, spectral gap lemmas \ref{Spectral Gap} and \ref{weaklycontained}, and Popa's transversality argument imply that the deformation $\alpha_t$ converges uniformly on $(Q^t)' \cap M$, and in particular on $P = p((Q^t)' \cap M)$. Now, by Theorem \ref{adapted IPV}, the position of $P$ is described by one of the following situations :
\begin{enumerate}
\item $\mathcal N_{pMp}(P)'' \prec_M A \rtimes \Sigma$, for some $\Sigma \in \mathcal{S}$ ;
\item $P \prec_M L\Gamma$. 
\end{enumerate}
But case (1) is impossible because $Q \subset \mathcal N_{pMp}(P)''$ has no amenable direct summand, whereas all the $A \rtimes \Sigma$, $\Sigma \in \mathcal{S}$ are amenable.
The moreover part is a consequence of the moreover part in Theorem \ref{adapted IPV}.
\end{proof}

\begin{proof}[Proof of theorem A]
As pointed out in the introduction, it is enough to show that if $P \subset A$ is a diffuse subalgebra, then $Q = P' \cap M$ is amenable. Hence consider $q \in \mathcal Z(Q)$ a maximal projection such that $qQ$ is amenable. Assume by contradiction that $q \neq 1$. Thus $(1-q)Q \subset (1-q)M(1-q)$ has no amenable direct summand, and Theorem \ref{positioncommutant} implies that $(1-q)(Q' \cap M) \prec_M L\Gamma$. Since $P \subset Q' \cap M$, we get $(1-q)P \prec_M L\Gamma$. This contradicts the fact that $P \subset A$ is diffuse.
\end{proof}

\begin{proof}[Proof of Theorem B]
As in the statement of the theorem, assume that the representation $\pi$ has a tensor power which is tempered, and that $\pi$ is mixing relative to a finite family $\mathcal{S}$ of amenable almost malnormal subgroups of $\Gamma$.\\
Consider a subalgebra $Q \subset M$ such that $Q \nprec_M L\Gamma$.\\

{\sc Step 1.} Construction of the projections $p_n$.\\
This is similar to the proof of \cite[Proposition 6]{CI}.
Naturally, take for $p_0$ the maximal projection in $\mathcal Z(Q)$ such that $p_0Q$ is hyperfinite. Let us show that $(1-p_0)\mathcal Z(Q)$ is discrete.\\
Otherwise one can find a projection $p \in \mathcal{Z}(Q)$ with $p \leq 1-p_0$ such that $p\mathcal Z(Q)$ is diffuse. But $pQ$ has no amenable direct summand, and the moreover part of proposition \ref{positioncommutant} implies that either $p(Q'\cap M) \prec_M \C1$ or $\mathcal N_pMp(p(Q'\cap M))'' \prec L\Gamma$.
The first case is excluded becuase $p\mathcal Z(Q)$ is diffuse. The second case would imply that $Q \prec_M L\Gamma$, which is impossible as well.\\
Thus we obtain (at most) countably many projections $(p_n)_{n \geq 0}$ such that $p_0Q$ is hyperfinite, and $p_nQ$ is a non-hyperfinite factor for all $n \geq 1$.\\

{\sc Step 2.} For any $n \geq 1$, $p_nQ$ does not have property Gamma and is prime.\\
An easy amplification argument implies that it is sufficient to show that any non-hyperfinite subfactor $N \subset M$ such that $N \nprec_M L\Gamma$ is non-Gamma and prime.

{\it Non Property Gamma.} Since $N \subset M$ has no amenable direct summand, spectral gap lemma \ref{Spectral Gap} implies that the deformation converges uniformly on $N' \cap (M)^\omega$ and {\it a fortiori} on $B = N' \cap N^\omega$. So we are in the situation of Theorem \ref{ultraproduct}.\\
Assume that $N$ has property Gamma. Then $B \neq \C1$, and a classical argument implies that $B$ is diffuse. Moreover, by definition of $B$, we have that $N \subset B' \cap M$. Since $N \nprec_M L\Gamma$ and $N$ is non-amenable, cases (3) and (4) case of Theorem \ref{ultraproduct} are not satisfied, so that $B \subset A^\omega \rtimes \Gamma$. We will derive a contradiction from this fact.\\
As in the proof of Proposition 7 in \cite{ozawa2003}, we can construct a sequence of $\tau$-independent commuting projections $p_n \in N$ of trace $1/2$, such that $(p_n) \in N' \cap N^\omega$, and if $C = \lbrace p_n \, \vert \, n \in \N \rbrace''$, then $C' \cap N$ is not amenable.\\
Now, take a non-zero projection $q \in \mathcal{Z}(C' \cap N)$ such that $qC'\cap N$ has no amenable direct summand. By Proposition \ref{positioncommutant}, we get that $qC \prec_M L\Gamma$.\\
At this point, remark that the sequence of unitaries $w_n = 2p_n - 1 \in \mathcal{U}(C)$ converges weakly to $0$, and that $(w_n) \in N' \cap N^\omega \subset A^\omega \rtimes \Gamma$. The following claim leads to a contradiction.\\
{\it Claim.} For all $x,y \in M, \lim_n \Vert E_{L\Gamma}(xqw_ny)\Vert_2 = 0$.\\
By Kaplanski's density theorem, and by linearity, it suffices to prove the claim for $x = au_h$, $y= bu_k$, for $a,b \in A$, $h,k \in \Gamma$. Write $qw_n = \sum_{g \in \Gamma} a_{n,g}u_g$ and let $\varepsilon > 0$. Since $(qw_n) \in A^\omega \rtimes \Gamma$, there exists $F \in \Gamma$ finite such that $$\Vert P_F(qw_n) - qw_n \Vert_2 < \frac{\varepsilon}{2\Vert a \Vert \Vert b \Vert}, \; \forall n \in \N.$$
Now we have :
\begin{align*}
\Vert E_{L\Gamma}(xP_F(qw_n)y)\Vert_2^2  = & \sum_{g \in F} \vert \tau(a\sigma_h(a_{n,g})\sigma_{hg}(b))\vert^2\\
 = & \sum_{g \in F} \vert \tau(\sigma_{h^{-1}}(a)qw_nu_g^*\sigma_g(b))\vert^2.
\end{align*}
This quantity can be made smaller than $\varepsilon^2/4$ for $n$ large enough, and we get that $\Vert E_{L\Gamma}(xqw_ny) \Vert_2 < \varepsilon$ for $n$ large enough. That proves the claim and gives the desired contradiction.

{\it B. Primeness.} If $N = N_1 \ootimes N_2$, then $N_1$ and $N_2$ are factors, and one of them, say $N_1$, is non-amenable. Hence Theorem \ref{adapted IPV} implies that $N_2 \prec_M \C1$ or $N \prec_M A \rtimes \Sigma$ for some $\Sigma \in \mathcal{S}$, or $N \prec_M L\Gamma$. The only possible case is that $N_2$ is not diffuse, and we see that $N$ is prime.
\end{proof}

\begin{rems} 1) For the part about property Gamma, there is a shorter way to show that if $N' \cap N^\omega \subset A^\omega \rtimes \Gamma$, then $N' \cap N^\omega = \C$. Assume that $x \in A^\omega \rtimes \Gamma$ is $N$ central. For all $g \in \Gamma \setminus \lbrace e \rbrace$, put $x_g = E_{A^\omega}(xu_g^*)$. We get for all $a \in A$, $ax_g = x_g\sigma_g(a)$, and so $x_g^*x_ga = x_g^*x_g\sigma_g(a)$. So if the action is free (\emph{i.e.} if $\pi$ is faithful), we get that $E_A(x_g^*x_g) = 0$. Thus $x \in A^\omega$. But by strong ergodicity, $N' \cap A^\omega = \C$. We thank Cyril Houdayer for this shorter proof. However we prefer to keep the proof of Theorem B as it is because it does not use the commutativity of $A$, which will be useful later.\\
2) The primeness result remains true if we replace the condition of being almost malnormal for the elements in $\mathcal S$, by being abelian. Indeed, in that case, write $N = N_1 \ootimes N_2$, with $N_1$ non-amenable. Then the second part of Theorem \ref{adapted IPV} does not apply, but by Theorem \ref{positioncommutant}, we get that $N_2 \prec_M L\Gamma$. Assume that $N_2$ is diffuse. Since it is a factor, $N_2 \nprec_M L\Sigma$, for all $\Sigma \in \mathcal S$. A modified version of lemma \ref{norm1} then gives $N \subset \mathcal{N}_M(N_2)'' \prec_M L\Gamma$, which contradicts $A \subset N$.
\end{rems}

\section{An adaptation to the case of Bogoliubov actions}

\subsection{Statement of the Theorem}

We first recall the main definitions on the CAR-algebra and Bogoliubov actions. We refer to chapters 7 and 8 in \cite{HJ} for a consistent material on this topic.

Consider a unitary representation $(\pi,H)$ of a discrete countable group $\Gamma$. Denote by $A(H)$ the CAR-algebra of $H$. By definition, $A(H)$ is the unique $C^*$-algebra generated by elements $(a(\xi))_{\xi \in H}$ such that :
\begin{itemize}
\item $\xi \mapsto a(\xi)$ is a linear map ;
\item $a(\xi)a(\eta) + a(\eta)a(\xi) = 0$, for all $\xi, \eta \in H$ ;
\item $a(\xi)a(\eta)^* + a(\eta)^*a(\xi) = \langle \xi , \eta \rangle$, for all $\xi, \eta \in H$.
\end{itemize}
Moreover, for any unitary $u \in B(H)$, one can define an automorphism $\theta_u$ of $A(H)$ by the formula $\theta_u(a(\xi)) = a(u\xi)$, and the map $\theta : \mathcal{U}(H) \rightarrow \Aut(A(H))$ is a continuous homomorphism for the ultra-strong topology on $\mathcal U(H)$ and the pointwise norm convergence topology in $\Aut(A(H))$.
Hence the representation $\pi$ gives rise to an action of $\Gamma$ on $A(H)$.\\
Now consider the quasi-free state $\tau$ on $A(H)$ associated to $1/2 \in B(H)$. By definition, $\tau$ is determined by the formula :
$$\tau(a(\xi_m)^* \cdots a(\xi_1)^*a(\eta_1) \cdots a(\eta_n)) = \frac{1}{2^n}\delta_{n,m}\det(\langle \xi_j , \eta_k \rangle ).$$
Then the von Neumann algebra $R_H$ on $L^2(A(H),\tau)$ generated by $A(H)$ is isomorphic to the hyperfinite II$_1$ factor and $\tau$ is the unique normalized trace on $R_H$. In addition the action of $\Gamma$ on $A(H)$ defined above extends to a trace preserving action on $R_H$, called the \emph{Bogoliubov action} associated to $\pi$. We recall the statement of the theorem that we will prove.

\begin{thmc}
Assume that the representation $\pi$ is mixing relative to a finite family $\mathcal{S}$ of almost-malnormal amenable subgroups of $\Gamma$ and has a tensor power which is tempered.
Consider the Bogoliubov action $\Gamma \curvearrowright R$ on the hyperfinite II$_1$ factor associated to $\pi$, and put $M = R \rtimes \Gamma$.\\
Let $Q \subset M$ be a subalgebra such that $Q \nprec_M L\Gamma$. Then there exists a sequence  $(p_n)_{n \geq 0}$ of projections in $\mathcal Z(Q)$ with $\sum_n p_n = 1$ such that :
\begin{itemize}
\item $p_0Q$ is hyperfinite ;
\item $p_nQ$ is a prime factor and does not have property Gamma.
\end{itemize}
\end{thmc}

To prove this theorem, we proceed as in the Gaussian case. It would be too heavy to reprove everything in details, so we just give the main steps and tools of the proof, hoping that this is enough to convince the reader.

\subsection{The deformation of $M$}
Denote by $M = R \rtimes \Gamma$, and put $\tilde M = \tilde{R} \rtimes \Gamma$, where $\Gamma$ acts on $\tilde R = R_{H \oplus H}$ by the Bogoliubov action corresponding to the representation $\pi \oplus \pi$.
Since $H = H \oplus 0 \subset H \oplus H$ one has $R \subset \tilde{R}$, and the action of $\Gamma$ on $R$ is the restriction of the action on $\tilde R$, so that $M \subset \tilde{M}$.\\
On $H \oplus H$ define $\theta_t$ and $\rho$ as in section \ref{notationsutiles}. These unitaries induce a $s$-malleable deformation $(\alpha_t,\beta)$ of $\tilde M$.\\

Before moving forward, we explain the main difference with the Gaussian case.\\
Note that in $\tilde M$ there is also a copy $R_0$ of $R$ coming from $0 \oplus H \subset H \oplus H$. However $\tilde R$ is certainly not isomorphic to the tensor product $R \ootimes R \simeq R \, \ootimes \, R_0$, because $R$ and $R_0$ do not commute to each other. To fix this problem we first check the following two facts, we will see later how to use it.
\begin{enumerate}[(i)]
\item The linear span of elements of the form $ab$ with $a \in R$, $b \in R_0$ forms an ultrastrongly dense subalgebra of $\tilde R$ ;
\item $R$ and $R_0$ are $\tau$-independent.
\end{enumerate}
Before proving the facts, we introduce some notations taken from \cite[Exercise XIV.5]{takesakiIII}. For a unitary representation $(\rho,K)$, denote by $\theta_{-1}$ the automorphism of $R_K$ induced by $-id \in \mathcal{U}(K)$, and put :
\begin{align*}
R_{K,\ev} & = \lbrace x \in R_K \, \vert \, \theta_{-1}(x) = x \rbrace ;\\
R_{K,\odd} & = \lbrace x \in R_K \, \vert \, \theta_{-1}(x) = -x \rbrace.
\end{align*}
Remark that $R_K = R_{K,\ev} \oplus R_{K,\odd}$. Now point (i) follows from the easily checked relations :
\begin{itemize}
\item[(iii)] $xy = yx$ for all $x \in R_{\ev}$, $y \in R_0$ ;
\item[(iv)] $xy = \theta_{-1}(y)x$, for all $x \in R_{\odd}$, $y \in R_0$.
\end{itemize}
To prove (ii), take $x \in R$, and $y \in R_0$. If $x \in R_{\ev}$, then $z \in R_0 \mapsto \tau(xz)$ is a trace on the factor $R_0$ so it is equal to $\tau(x)\tau$, and we indeed get $\tau(xy) = \tau(x)\tau(y)$. If $x \in R_{\odd}$, then $\tau(xy) = \tau(\theta_{-1}(y)x) = \tau(y\theta_{-1}(x)) = - \tau(yx) = 0$. But it is also true if $y = 1$ : $\tau(x)=0$. Hence, in that case too, we get $\tau(xy) = \tau(x)\tau(y)$. By linearity, this relation is true for any $x \in R_H$.

\subsection{Adaptation of the main ingredients and sketch of proof}
We first check that the 2 main ingredients of the proof can be adapted : the spectral gap lemma and the mixing properties of the action.

\begin{lem}[Spectral gap argument] If $\pi$ is tempered, the following relation between $M$-$M$ bimodules is true. $$L^2(\tilde{M}) \ominus L^2(M) \subset_w L^2(M) \otimes L^2(M)$$
\end{lem}
\begin{proof}
We first show an intermediate step.\\
{\it Claim.} If $\pi$ is weakly contained in the regular representation $\lambda$, so is the representation $\sigma = \sigma_\pi^0$ on $L^2(R_H) \ominus \C$ induced by the Bogoliubov action of $\pi$.\\
To prove this claim, we need to check that for all $\xi, \eta \in L^2(R_H) \ominus \C$, the coefficient function $\varphi_{\xi,\eta} : g \mapsto \langle \sigma_g(\xi),\eta \rangle$ can be approximated on finite subsets of $\Gamma$ by sums of coefficient functions of $\lambda$. This will be denoted $\varphi_{\xi,\eta} \preceq \lambda$.
By a linearity/density argument, we can assume that
$$\xi = a(\xi_n)^* \cdots a(\xi_1)^*a(\eta_1)\cdots a(\eta_n) \text{ and }  \eta = \eta_0 - \tau(\eta_0),$$
with $\eta_0 = a(\xi'_1) \cdots a(\xi'_l)a(\eta'_k)^*\cdots a(\eta'_1)^*$. Indeed in that case we will get that $\varphi_{h,k} \prec \lambda$ for all $h \in L^2(R_H)$, $k \in L^2(R_H) \ominus \C$, and in particular for $h \in L^2(R_H) \ominus \C$.
But a computation gives
\begin{align*}
\langle \sigma_g(\xi) , \eta \rangle = & 1/2^{n+l} \delta_{n+l,m+k} \det \begin{pmatrix} \langle \sigma_g(\eta_i), \xi'_j \rangle & \langle \eta_i, \xi_j \rangle \\ \langle \eta'_i, \xi'_j \rangle & \langle \eta'_i, \sigma_g(\xi_j) \rangle \end{pmatrix}\\
& - 1/2^n \delta_{n,m}\det(\langle \eta_i , \xi_j \rangle)1/2^l\delta_{k,l}\det(\langle \eta'_i, \xi'_j \rangle).
\end{align*}
Developing the above determinant, we get a linear combination of terms that can be approximated by coefficient of finite tensor powers of $\lambda$, plus a term equal to $\det(\langle \eta'_i, \xi'_j \rangle)\det(\langle \eta_i,\xi_j \rangle)$ if $n = m$, $k = l$, and 0 otherwise. Therefore this extra-term cancels with the second term of the above equality.\\
So by Fell's lemma we get $\varphi_{\xi,\eta} \prec \oplus_n \lambda^{\otimes n} \prec \lambda$, which proves the claim.

Now we can prove the lemma.
First, using the facts (i)-(iv) of the previous subsection one easily checks the isomorphism of $M$-$M$ bimodules
$$L^2(\tilde{M}) \ominus L^2(M) \simeq \mathcal{H}_1 \oplus \mathcal{H}_2,$$
with $\mathcal{H}_1 = L^2(R) \otimes L^2(R_{0,\ev}) \ominus \C \otimes l^2(\Gamma)$ and $\mathcal{H}_2 = L^2(R,\theta_{-1}) \otimes L^2(R_{0,\odd}) \otimes l^2(\Gamma)$, and the bimodule structures on $\mathcal H_1$ and $\mathcal H_2$ given respectively by $$au_g(x \otimes \xi \otimes \delta_h)bu_k = a\sigma_g(x)\sigma_{hg}(b) \otimes \sigma_g(\xi) \otimes \delta_{ghk},$$
$$au_g(x \otimes \eta \otimes \delta_h)bu_k = a\sigma_g(x)\theta_{-1}(\sigma_{hg}(b)) \otimes \sigma_g(\eta) \otimes \delta_{ghk},$$ 
$a, x, b \in R$, $\xi \in  L^2(R_{0,\ev})\ominus \C$, $\eta \in  L^2(R_{0,\odd})$, $g, h, k \in \Gamma$.\\
We have to show that $\mathcal{H}_i \subset_w L^2(M) \otimes L^2(M)$, for $i = 1,2$. We do it only for $\mathcal H_2$, the case of $\mathcal H_1$ being similar. By the claim above, we get that $$\mathcal{H}_2 \subset_w L^2(R,\theta_{-1}) \otimes \ell^2(\Gamma) \otimes \ell^2(\Gamma),$$ with an $M$-$M$ bimodule structure similar to the one on $\mathcal H_2$.\\
In fact, this last bimodule is seen to be isomorphic to $L^2(M) \otimes_R L^2(R,\theta_{-1}) \otimes_R L^2(M)$, and since $R$ is amenable, we also have that $L^2(R,\theta_{-1}) \subset_w L^2(R) \otimes L^2(R,\theta_{-1})$.
But $L^2(R) \otimes L^2(R,\theta_{-1})$ is clearly isomorphic to the coarse $R$-$R$ bimodule.\\
In summary, we get that $\mathcal H_2 \subset_w L^2(M) \otimes L^2(M)$.
\end{proof}
With a little more care, one could show that the conclusion of the above lemma remains true if one just assume that some tensor power of $\pi$ is tempered.

\begin{lem}[mixing property] Assume that $\pi$ is mixing relative to a family $\mathcal S$ of subgroups of $\Gamma$. Then the associated Bogoliubov action $\sigma_\pi$ is mixing relative to $\mathcal S$ as well.
\end{lem}
\begin{proof}
This is Theorem 2.3.2(1) of \cite{CCJJV} in the mixing case. The relative mixing case is treated in the same way.
\end{proof}

Now using the mixing property and relations (i)-(iv) of the previous subsection, one can imitate line by line the proof of Lemma 3.8 in \cite{vaesonecoho}, so that Lemma \ref{norm1} and Lemma \ref{technicallemma} can be adapted to this context. 

Hence, all the material needed to prove Theorem \ref{adapted IPV} (the one about the position of subalgebras in $M$) admits an analogous in the setting of Bogoliubov actions, so that this holds true for these actions (under the same assumptions on $\pi$). The reader may have noticed that there is a step in the proof that needs to be adapted\footnote{namely, one needs to check that if $\alpha_1(Q) \prec_{\tilde M} Q$ then $Q \prec_M L\Gamma$.}, but it all works thanks to the properties (i)-(iv) of the previous subsection.

Now Theorem \ref{ultraproduct} (the one about the position of subalgebras in $M^\omega$) and then Theorem C can be proven exactly as in the Gaussian case.

%\addcontentsline{toc}{section}{Bibliography}
\bibliographystyle{alpha1}

\end{document}